\documentclass[12pt]{article}%字体大小
\usepackage{graphicx} %% Package for Figure
\usepackage{subcaption}
\usepackage{float} %% Package for Float
\usepackage{amssymb}
\usepackage{amsmath}
\usepackage{enumitem} % 加载 enumitem 宏包
\usepackage{mathtools}
\usepackage[thmmarks,amsmath]{ntheorem} %% If amsmath is applied, then amsmath is necessary
\usepackage{lineno}
\usepackage{xcolor}
\definecolor{darkblue}{RGB}{0, 0, 100}
\definecolor{darkgreen}{RGB}{0, 100, 0}

\usepackage[numbers]{natbib}  % 强制使用数字式引用
\usepackage{bm} %% Bold Mathematical Symbols
\usepackage{extarrows}
\usepackage{mathrsfs} %% Swash letter

\usepackage{lipsum}


\usepackage{booktabs} %% tables with three lines
{%% Environment of Proof
    \theoremstyle{nonumberplain}
    \theoremheaderfont{\bfseries}
    \theorembodyfont{\normalfont}
    \theoremsymbol{\mbox{$\Box$}}
    \newtheorem{proof}{Proof}
}
\makeatletter
\newcounter{mycounter}[section]

{%% 定义环境
    \theoremstyle{plain}
    \theoremheaderfont{\bfseries}
    \theorembodyfont{\normalfont}
    \newtheorem{theorem}[mycounter]{Theorem}
    
    \newtheorem{lemma}[mycounter]{Lemma}
    \newtheorem{definition}[mycounter]{Definition}

    \newtheorem{fact}[mycounter]{Fact}
    \newtheorem{claim}[mycounter]{Claim}
    
    \newtheorem{property}[mycounter]{Property}
    
    \newtheorem{conjecture}[mycounter]{Conjecture}
}

{%% Environment of Remark
    \theoremheaderfont{\bfseries}
    \theorembodyfont{\normalfont}
    
}

\makeatother

\newcommand{\RNum}[1]{\uppercase\expandafter{\romannumeral #1\relax}}
\allowdisplaybreaks[4]

\graphicspath{{figure/}}                                 %% Path of Figures
\usepackage[a4paper]{geometry}                        %% Paper size
\begin{document}
\title{{Proof of the linkage conjecture for highly connected tournaments} \thanks{The author's work is supported by National Natural Science Foundation of China (No.12071260)}}

\author{Jia Zhou, Jin Yan\thanks{Corresponding author. E-mail adress: yanj@sdu.edu.cn}  \unskip\\[2mm]
School of Mathematics, Shandong University, Jinan 250100, China}

\date{}
\maketitle

%MS+++++++++++++++++++++ Abstract +++++++++++++++++++++++++

\begin{abstract}
A digraph $D$ is $k$-linked if for every $2k$ distinct vertices $ x_1,\ldots , x_k, y_1, \ldots , y_k$ in $D$, there exist $k$ pairwise vertex-disjoint paths $P_1,\ldots, P_k$ such that $P_i$ starts at $x_i$ and ends at $y_i$ for each $i\in [k]$. In 2021, Gir\~{a}o, Popielarz, {and Snyder [Combinatorica 41 (2021) 815--837] conjectured} that there exists a constant $C >0$ such that every $(2k+1)$-connected tournament with minimum out-degree at least $Ck$ is $k$-linked. In this paper, we disprove this conjecture by constructing a family of counterexamples with minimum out-degree at least $\frac{k^2+11k}{26}$ (for $k\geq 42$). Further, we prove that every $(2k+1)$-connected semicomplete digraph $D$ with minimum out-degree at least $ 7k^2 + 36k$ is $k$-linked. This result is optimal in terms of both connectivity and minimum out-degree (up to a multiplicative factor), which refines and generalizes the earlier result of Gir\~{a}o, Popielarz, and Snyder.
\end{abstract}		

\vspace{1ex}
\vspace{1ex}


\section{Introduction}
	Connectivity is one of the most fundamental concepts in graph theory. Let \( k \) be a positive integer. A (di)graph is \textbf{\( k \)-connected} if it has at least \( k+1 \) vertices and, after deleting any set \( S \) of at most \( k-1 \) vertices, there exists a path from \( x \) to \( y \) for any pair of distinct vertices \( x \) and \( y \). Although Menger's Theorem characterizes  vertex-disjoint paths in \(k\)-connected graphs, applications often demand that these paths specifically link predetermined pairs of start and end vertices, which motivates the definition of \(k\)-linkedness. A (di)graph \( D \) is \textbf{\( k \)-linked} if for every \( 2k \) distinct vertices \( x_1, \ldots, x_k, y_1, \ldots, y_k \) in \( D \), there exist \( k \) pairwise vertex-disjoint paths \( P_1, \ldots, P_k \) such that \( P_i \) starts at \( x_i \), ends at \( y_i \) {for each \( i \in [k] \)}. Clearly, every \( k \)-linked graph is \( k \)-connected. The converse, however, seems far from true.

	The relationship between connectivity and \(k\)-linkedness in graphs and digraphs has long been a central theme of research. Early contributions came from Larman and Mani \cite{ Larman(1974)}, as well as Jung \cite{ jung1970}, who demonstrated that there exists a function \(f(k)\) such that any \(f(k)\)-connected graph is \(k\)-linked. Subsequently, Bollob\'{a}s and Thomason \cite{Bollobas(1996)} established that \(f(k) = 22k\) suffices, leveraging this result to confirm a conjecture of Mader \cite{Mader(1996)} and another of Erd\H{o}s and Hajnal \cite{Erd(1969)} concerning the smallest average degree required to guarantee a subdivision of a clique on \(k\) vertices (Koml\'{o}s and Szemer\'{e}di \cite{Komlos(1996)} independently proved this as well). Later, Thomas and Wollan \cite{Thomas(2005)} refined the constant in \(f(k)\), showing that every \(2k\)-connected graph {with minimum degree \(10k\)} is \(k\)-linked, which is the best bound currently. In stark contrast, Thomassen \cite{Thomassen(1991)} dashed the hope of directly generalizing the relevant conclusions (about graphs) to digraphs by constructing digraphs with arbitrarily high connectivity that fail 2-linkedness. This {also disproved} his own conjecture \cite{Thomassen(1980)} that there exists a function \(f(k)\) such that every \(f(k)\)-connected digraph is \(k\)-linked. 

As a special class of digraphs with rich structural properties, \textbf{tournaments}, which are defined as digraphs obtained by orienting each edge of a complete graph in exactly one direction, reveal a starkly different scenario. They have thus become central objects in studying the relationship between connectivity and linkedness. Pioneering work by Thomassen \cite{Thomassen(1984)} showed the existence of a function \(g(k)\) such that every \(g(k)\)-connected tournament is \(k\)-linked. However, the initial bound for \(g(k)\) was exponential in \(k\). {In 2014, K\"{u}hn, Lapinskas, Osthus, and Patel \cite{Khn(2014)} reduced the bound to \(g(k) = O(k \log k)\) and conjectured that a linear bound in \(k\) might be feasible.} Pokrovskiy \cite{Pokrovskiy(2015)} confirmed this conjecture in {2015}, proving that \(g(k)=452k\). This landmark result established a linear connectivity threshold for \(k\)-linked tournaments. Inspired by the result in \cite{Thomas(2005)} on graphs, Pokrovskiy proposed the following conjecture. Let $\delta^+(D)$ (resp., $\delta^-(D)$) denote the \textbf{minimum out-degree} (resp., the \textbf{minimum in-degree}) of $D$, and let $\delta^0(D)=\min \{\delta^+(D), \delta^-(D)\} $ denote the \textbf{minimum semidegree} of $D$.% The 

\begin{conjecture}\label{conj1}
	\cite{Pokrovskiy(2015)} For every $k\in \mathbb{N}$, there exists an integer $h(k)$ such that every $2k$-connected tournament with $\delta^0(D)\geq h(k)$ is $k$-linked.
\end{conjecture}

The authors \cite{Zhou(2025)} constructed a family of \( 2k \)-connected tournaments of order \( n \geq 14k^2 \) with minimum semidegree at least \( \left\lfloor \frac{n - 2k}{2k + 2} \right\rfloor \), which are not \( k \)-linked. {This result, which disproves Conjecture \ref{conj1}, indicates that  connectivity \( g(k) \geq 2k + 1 \) is necessary when the minimum semidegree is a function of \( k \).} Several results related to Conjecture \ref{conj1} have been studied, including those in \cite{Bang(2021), Snyder(2019), Meng(2021), Zhou(2025), Zhou(2023)}. Notably, Gir\~{a}o, Popielarz, and Snyder \cite{Girao(2021)} established the following theorem. 

\begin{theorem}\label{the2}
\cite{Girao(2021)} Every $(2k + 1)$-connected tournament with $\delta^+(T)\geq Ck^{31}$ is $k$-linked for some constant $C$.
\end{theorem}

In the same paper, they also raised the following conjecture.
\begin{conjecture}\label{conj2}
\cite{Girao(2021)} There exists a constant $C>0$ such that every $(2k+1)$-connected tournament with $\delta^+(T)\geq Ck$ is $k$-linked.
\end{conjecture}

Bang-Jensen \cite{Bang(1989)} showed that any 5-connected tournament is 2-linked. This implies that when \( k = 2 \), Conjecture \ref{conj2} holds even without additional out-degree {condition}. However, in this paper, we provide a negative answer to Conjecture \ref{conj2} when {$k$} is relatively large.

\begin{theorem}\label{main1}
For every integer $ k \geq 42 $ and integer $ n \geq k^2 $, there exists a family of $(2k+1)$-connected tournaments on $ n $ vertices with minimum out-degree at least $\frac{k^2 + 11k}{26}$ that are not $k$-linked.
\end{theorem}

Furthermore, we prove Theorem \ref{main2}, which establishes an optimal result in terms of both connectivity and minimum out-degree (up to a multiplicative factor) for semicomplete digraphs. A digraph $D$ is \textbf{semicomplete} if there is at least one arc between every pair of vertices in $D$. Since tournaments are a strict subclass of semicomplete digraphs, Theorem \ref{main2} simultaneously refines and generalizes Theorem \ref{the2}.

\begin{theorem}\label{main2}
Every $(2k+1)$-connected semicomplete digraph $D$ with $\delta^+(D)\geq  7k^2 + 36k$ is $k$-linked.
\end{theorem}

The minimum out-degree bound in Theorem \ref{main2} is substantially lower than \( Ck^{31} \) required in Theorem \ref{the2}. {In the proof of Theorem \ref{main2}, a novel strategy is employed. A powerful tool in the proof of Theorem \ref{main2} is the nearly in-dominating set we proposed (Definition \ref{def2}), which acts as a `connector' linking two disjoint path systems. The first path system is derived using Menger's Theorem. In contrast, to generate the second path system, we present a innovative iterative path adjustment program that releases some vertices {from the first path system.} This program relies on our key discovery: every tournament possesses a nearly out-dominating vertex (Lemma \ref{key1}). When combined with a sufficiently large out-degree, this iterative adjustment process enables the release of enough vertices to construct the second path system.}

The rest of the paper is organized as follows. Section 2 provides the notation used throughout the article. Section 3 presents the counterexamples in Theorem \ref{main1}. In Section 4, we first introduce some fundamental tools and definitions, and then proceed to the proof of Theorem \ref{main2}. %Finally, we compile several open problems in Section 5.On the other hand, Bang-Jensen and Jord\'{a}n \cite{Bang(2010)} conjectured that every \( (2k - 1) \)-connected semicomplete digraph contains a spanning \( k \)-connected tournament. This conjecture remains open for \( k \geq 4 \). If proven true, their result would imply that semicomplete digraphs require relatively high connectivity (specifically, approximately twice the connectivity of tournaments) to guarantee \( k \)-linkedness. Nevertheless, Theorem \ref{main2} demonstrates that even under the same connectivity conditions as tournaments, semicomplete digraphs can be guaranteed to be \( k \)-linked.

\section{Basic terminology}
Notation not introduced here is consistent with \cite{Bang(2009)}. Let $\mathbb{N}$ be the set of natural numbers. For an integer $i$, we use the notation  $\boldsymbol{[i]} = \{1, \ldots, i\}$, and $\boldsymbol{[i, i+j]} = \{i, i+1, \ldots, i+j\}$. The digraphs considered in this paper are always simple, i.e., without loops and multiple arcs. Let $D$ be a digraph with vertex set $V(D)$ and arc set $A(D)$. We use $\boldsymbol{|D|}$ to represent the number of vertices in $D$. For two vertices $x, y \in V(D)$, we denote the arc from $x$ to $y$ as $xy$. For a vertex $x$ of $D$, we define $\boldsymbol{N^+_D(x)} = \{y \mid xy \in A(D)\}$ (resp. $\boldsymbol{N^-_D(x)} = \{y \mid yx \in A(D)\}$) as the \textbf{out-neighborhood} (resp. \textbf{in-neighborhood}) of $x$, and $\boldsymbol{d^+_D(x)} = |N^+_D(x)|$ (resp. $\boldsymbol{d^-_D(x)} = |N^-_D(x)|$) as the \textbf{out-degree} (resp. \textbf{in-degree}) of $x$.  %Furthermore, let $\delta^+(D)$ (resp. $\delta^-(D)$) denote the minimum out-degree (resp. minimum in-degree) of $D$.The minimum of $\delta^+(D)$ and $\delta^-(D)$ is called the \textbf{minimum semi-degree} of $D$ and is denoted by $\delta^0(D)$.%
For a set $X \subseteq V(D)$, the subgraph of $D$ induced by $X$ is denoted by $\boldsymbol{D\langle X \rangle}$, and the digraph obtained from $D$ by deleting $X$ and all arcs incident with $X$ is denoted by $\boldsymbol{D \setminus X}$. For a digraph $D$ and sets $A, B \subseteq V(D)$, we use the notation \(\boldsymbol{A \Rightarrow B}\) if and only if for every \(a \in A\) and \(b \in B\), there is an arc \(a b\). In particular, $\boldsymbol{x}$ \textbf{dominates} $\boldsymbol{y}$ if $x \rightarrow y$. %For a digraph $ D $, we denote by $\boldsymbol{\kappa(D)} $ the degree of strong connectivity of $ D $. So $\kappa(D)=k$ means that $D$ is $k$-connected but not $(k+1)$-connected.

A \textbf{matching} in a digraph $D=(V,A)$ is a set of arcs $M=\{u_1v_1,u_2v_2,\ldots{},u_rv_r\}\subset A$  such that no two arcs in $M$ have a common vertex. We say that $M$ is a \textbf{perfect matching} from $\{u_1,u_2,\ldots ,u_r\}$ to $\{v_1,v_2,\ldots ,v_r\}$. Suppose below that $P = x_1x_2 \cdots x_t$ is a directed path of $D$. We say that  $x_1$ (resp. $x_t$) is  the \textbf{initial} (resp. \textbf{terminal}) vertex of $P$ and that $P$ is an $(x_1,x_t)$-path. The \textbf{length} of $P$ is the number of arcs, and we denote a path of length $l$ as an \textbf{$\boldsymbol{l}$-path}. We denote by  \textbf{$\boldsymbol{P[x, y]}$} the subpath of $P$ from $x$ to $y$. Furthermore, $P$ is \textbf{minimal} if, for every $(x_1, x_t)$-path $Q$, either $V(P) = V(Q)$ or $V(Q)$ is not properly contained in $V(P)$. Suppose that \(Q = x_t x_{t + 1}\cdots x_m\) is a path in \(D\) that is internally disjoint from $P$, and we denote the \textbf{concatenation} of \(P\) and \(Q\) by \(\boldsymbol{P \circ Q}\), which is the directed path $x_1x_2\cdots x_tx_{t+1}\cdots x_m$. Given a collection of paths $\mathcal{Q}$, we use $\boldsymbol{\text{\textbf{Init}} (\mathcal{Q})}$ to represent the set of all initial vertices of paths in $\mathcal{Q}$, $\boldsymbol{\text{\textbf{Ter}} (\mathcal{Q})}$ to represent the set of all terminal vertices of paths in $\mathcal{Q}$, and $\boldsymbol{\text{\textbf{Int}}(\mathcal{Q})} = V(\mathcal{Q}) \setminus (\text{Ini} (\mathcal{Q}) \cup \text{Ter} (\mathcal{Q}))$. Also, a \(\boldsymbol{(u,X)}\)-path is a path that starts at vertex \(u\) and ends at a vertex in the set \(X\). Hereafter, \textbf{disjoint paths} denote vertex-disjoint paths. 

We say that a digraph is \textbf{regular} if all vertices have the same in-degree and out-degree. The unique acyclic tournament of order \( n \) is the \textbf{transitive tournament}, denoted by \( TT_n \). \textbf{A transitive tournament with vertex order} \( (v_1, v_2, \ldots, v_n) \) is a tournament such that for all \( 1 \leq i < j \leq n \), there is an arc from \( v_i \) to \( v_j \). A digraph is called an \textbf{oriented graph} if it has no pair of arcs of the form $xy$, $yx$. A \textbf{bipartite tournament} \( T[U,W] \) is defined on \( V=U\cup W \) such that for every pair of vertices \( u \in U \) and \( w \in W \), exactly one of the arcs \( xy \) or \( yx \) exists in \( T[U,W]  \), and no arcs exist within \( U \) or within \( W \).

\section{Counterexamples}

To construct the counterexamples in Theorem \ref{main1}, we begin by introducing a key fact.

\begin{fact}\label{fac1}
\cite{Li(2008)} Every regular tournament of order $n$ is $\lfloor n/3 \rfloor$-connected.
\end{fact}

%Let $l$ be a positive integer, and let $D$ be a digraph with a vertex $v$ and a subset $S \subseteq V(D)$. We say $(S, v, l, D)$ holds if there exist $l$ internally disjoint paths from $v$ to $S$ in $D$ (i.e., paths intersecting only at $v$). %Similarly, $(S, v, l, D)$ denotes $l$ internally disjoint paths from $S$ to $v$ in $D$.

%\begin{fact}\label{fac2} Let $D_1$ be a subdigraph of $D$, and let $S \subseteq V(D)$ be a subset with $|S| \leq 2k-1$. Suppose that for any two distinct vertices $x, y \in V(D_1) \setminus S$, there exists an $(x, y)$-path in $D \setminus S$. If $(v, V(D_1), 2k, D)$ and $(V(D_1), v, 2k, D)$ hold for some vertex \( v \in V(D \setminus D_1) \), then for any two vertices \( u, w \in (V(D_1) \cup \{v\}) \setminus S \), there exists a \((u, w)\)-path in \( D \setminus S \).\end{fact}
%\begin{proof}Let \(u, w \in (V(D_1 )\cup \{v\}) \setminus S\). If \(u, w \in V(D_1)\), the existence of a \((u, w)\)-path follows directly from the hypothesis. Assume without loss of generality that \(u = v\). By $(v, V(D_1), 2k, D)$, there exists a path $P$ from $v$ to some $x' \in V(D_1) \setminus S$ in $D \setminus S$. Since $x', w \in V(D_1 \setminus S)$, there exists an $(x', w)$-path $Q$ in $D \setminus S$. Concatenating $P$ and $Q$ yields a $(v, w)$-path. The case $w = v$ is symmetric by reversing the roles of $u$ and $w$.\end{proof}

\begin{proof}\textbf{of Theorem \ref{main1}.} Let \( k\geq 42 \) be a positive integer, and let $l=\lfloor k/13 \rfloor$. 
The counterexample tournament $T$ is constructed in three stages. First, we build a foundational oriented graph \( G \) whose structure is designed to avoid \( k \)-linkedness. Next, \( G \) is expanded into a tournament \( T_1 \) by adding extra vertices and completing all non-adjacent vertex pairs with arcs. This step facilitates achieving \((2k+1)\)-connectivity while preserving the critical structural properties. Finally, the tournament \( T \) is constructed by augmenting \( T_1 \) with a vertex set \( X \) and a vertex \( z^+ \). Throughout this process, the structure of \( G \) is leveraged to disrupt \( k \)-linkedness, while the \((2k + 1)\)-connectivity is achieved.  

To construct the oriented graph \( G \), we first design a path system \(\mathcal{P} = \{P_1, P_2, \ldots, P_k\}\) consisting of \(k\) disjoint paths, and each path is of the form \(P_i: z_i^0 \to z_i^1 \to \cdots \to z_i^{l} \to z_i^{l+1}\). For each \( t \in [0, l+1] \), define \( H^t = \{z_1^t, z_2^t, \ldots, z_{\lfloor k/2 \rfloor}^t\} \), ensuring \( |H^t| = \lfloor k/2 \rfloor \). Let \( U_2^+ = \bigcup^{\lfloor k/2 \rfloor}_{i=1} \text{Int}(P_i) \) and \( U_1^+ = \bigcup_{i=\lfloor k/2 \rfloor+1}^{k} \text{Int}(P_i) \), resulting in \( |U_1^+|, |U_2^+| \geq \lfloor \frac{k}{2} \rfloor \cdot l \geq (\frac{k}{2}-1)(\frac{k}{13}-1)$ $\geq \frac{k^2-15k+26}{26} \). To disrupt the \(k\)-linkedness, we augment \( \mathcal{P} \) by adding arcs such that the following holds (as shown in Fig. \ref{fig1}).
\begin{itemize}[itemsep=0pt, topsep=3pt, parsep=2pt]
    \item[(A1)] For each \( t \in [l] \), \( H^t \) induces a transitive tournament with vertex order \( (z_1^t, z_2^t, \ldots, z_{\lfloor k/2 \rfloor}^t) \).
     \item[(A2)]\( U_2^+ \) induces a tournament, where all arcs not in \(\mathcal{P}\) and not entirely within some \( H^j \) are directed from \( H^t \) to \( H^j \) for \( 1 \leq j < t \leq l \) (i.e., there is no arc from $H^j$ to $H^t$ with $t\geq j+2$).
    \item[(A3)] \( U_1^+ \) induces a tournament and \( U_2^+ \Rightarrow U_1^+ \) (i.e., $ U_2^+ \cup U_1^+$ induces a tournament).
    \item[(A4)] \( \text{Ter}(\mathcal{P}) \) forms a transitive tournament with vertex order \( (z^{l+1}_1, \ldots, z^{l+1}_k) \). The set \( U_1^+ \cup U_2^+ \cup \text{Ter}(\mathcal{P}) \) induces a tournament, where all arcs between \( U_1^+ \cup U_2^+ \) and \( \text{Ter}(\mathcal{P}) \) that are not in \( \mathcal{P} \) are directed from \( \text{Ter}(\mathcal{P}) \) to \( U_1^+ \cup U_2^+ \). Meanwhile, the subgraph \( G\langle U_1^+ \cup U_2^+, \text{Init}(\mathcal{P})  \rangle\) forms a bipartite tournament where all non-\(\mathcal{P}\) arcs directed from \( U_1^+ \cup U_2^+ \) to \( \text{Init}(\mathcal{P}) \). Also, \( \text{Ter}(\mathcal{P}) \Rightarrow \text{Init}(\mathcal{P}) \).
\end{itemize}
Notice that if \( G\langle\text{Init}(\mathcal{P}) \rangle \) can be made a tournament by adding arcs, then \( G \) is a tournament. 

\begin{figure}[H]
\centering    %居中       %子图居中
   \includegraphics[scale=0.45]{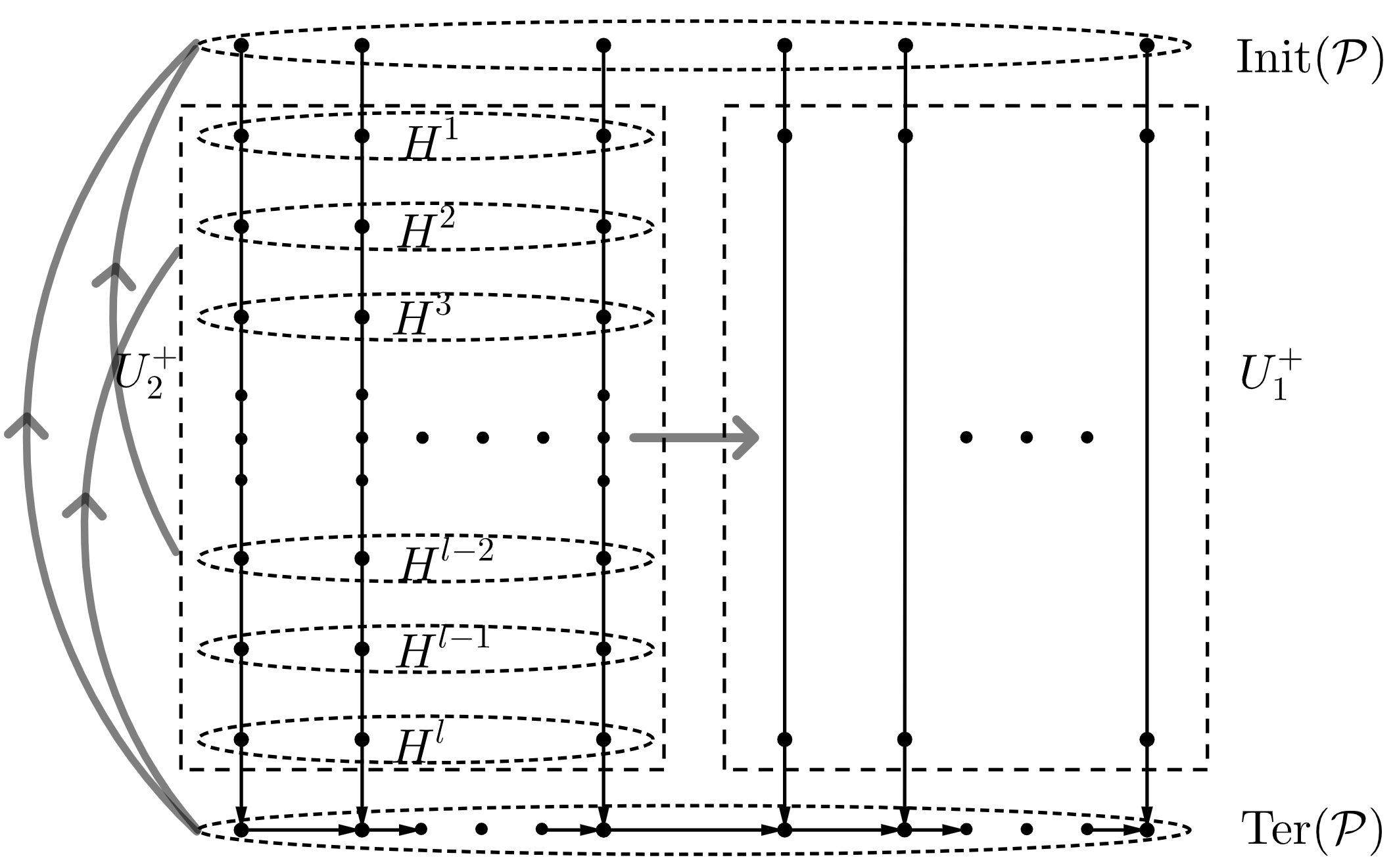}   %以pic.jpg的0.5倍大小输出
\caption{The structure of oriented graph \(G\). }
\label{fig1}
\end{figure}
\begin{property}\label{prop1}
There is no collection $\mathcal{R}_1$ of $\lfloor k/2 \rfloor-6$ disjoint paths from $H^0 $ to $H^{l+1}$ in $G\setminus \{z^{l+1}_{\lfloor k/2\rfloor+1},\ldots,z^{l+1}_{k}\}$, and a set $X^+_2\subseteq U_2^+$ of order $\lceil k/2 \rceil$ such that $V(\mathcal{R}_1)\cap X^+_2=\emptyset$.
 \end{property}
 \begin{proof}
Conversely, assume that such a collection $\mathcal{R}_1$ and set $X^+_2 \subseteq U_2^+$ exist. According to constructions (A2) and (A4), there is no arc from \(H^i\) to \(H^j\) if \(j \geq i + 2\) and $i, j \in [0,l+1]$. Since $U^+_2 \Rightarrow U^+_1$, by our construction, any path from \(H^0\) to \(H^{l+1}\) in $G\setminus \{z^{l+1}_{\lfloor k/2\rfloor+1},\ldots,z^{l+1}_{k}\}$ must intersect each of the sets \(H^1, \ldots, H^{l}\). So the path must intersect \(U^+_2\) at least \(l\) vertices, i.e., $|V(\mathcal{R}_1)\cap U^+_2|\geq (\lfloor k/2 \rfloor-6)\cdot l$. Because $|U^+_2|=\lfloor k/2 \rfloor \cdot l$, it follows that $|U^+_2\setminus V(\mathcal{R}_1)|\leq 6\lfloor \frac{k}{13} \rfloor<\lceil k/2 \rceil$, which contradicts the existence of the set \(X^+_2\).
 \end{proof}

 To construct the tournament \( T_1 \) from the oriented graph \( G \), we extend it as follows. The vertex set of \( T_1 \) is \( V(T_1) = V(G) \cup U^- \cup Y^- \cup Y \cup Y^+ \), where
\begin{itemize}[itemsep=0pt, topsep=3pt,parsep=2pt]
    \item the set \( U^- \) induces a regular tournament of order at least \( k^2/2\) with \( U^- \cap V(G) = \text{Init}(\mathcal{P}) \).
    \item \( Y^- = \{y_1^-, y_2^-, \ldots, y_k^-\} \), \( Y = \{y_1, y_2, \ldots, y_k\} \), and \( Y^+ = \{y_1^+, y_2^+, \ldots, y_k^+\} \). Specifically, \( Y^- \) induces a transitive tournament with vertex order \( (y_1^-, y_2^-, \ldots, y_k^-) \); \( Y \) induces a transitive tournament with vertex order \( (y_1, y_2, \ldots, y_k) \); and \( Y^+ \) induces a transitive tournament with vertex order \( (y_k^+, y_{k - 1}^+, \ldots, y_1^+) \) (note that the ordering of \( Y^+ \) is the reverse of the orderings of $H^t$ with $t\in [l]$, \( Y \), and \( Y^-\)).
\end{itemize}
\begin{figure}[H]
\centering    %居中       %子图居中
   \includegraphics[scale=0.45]{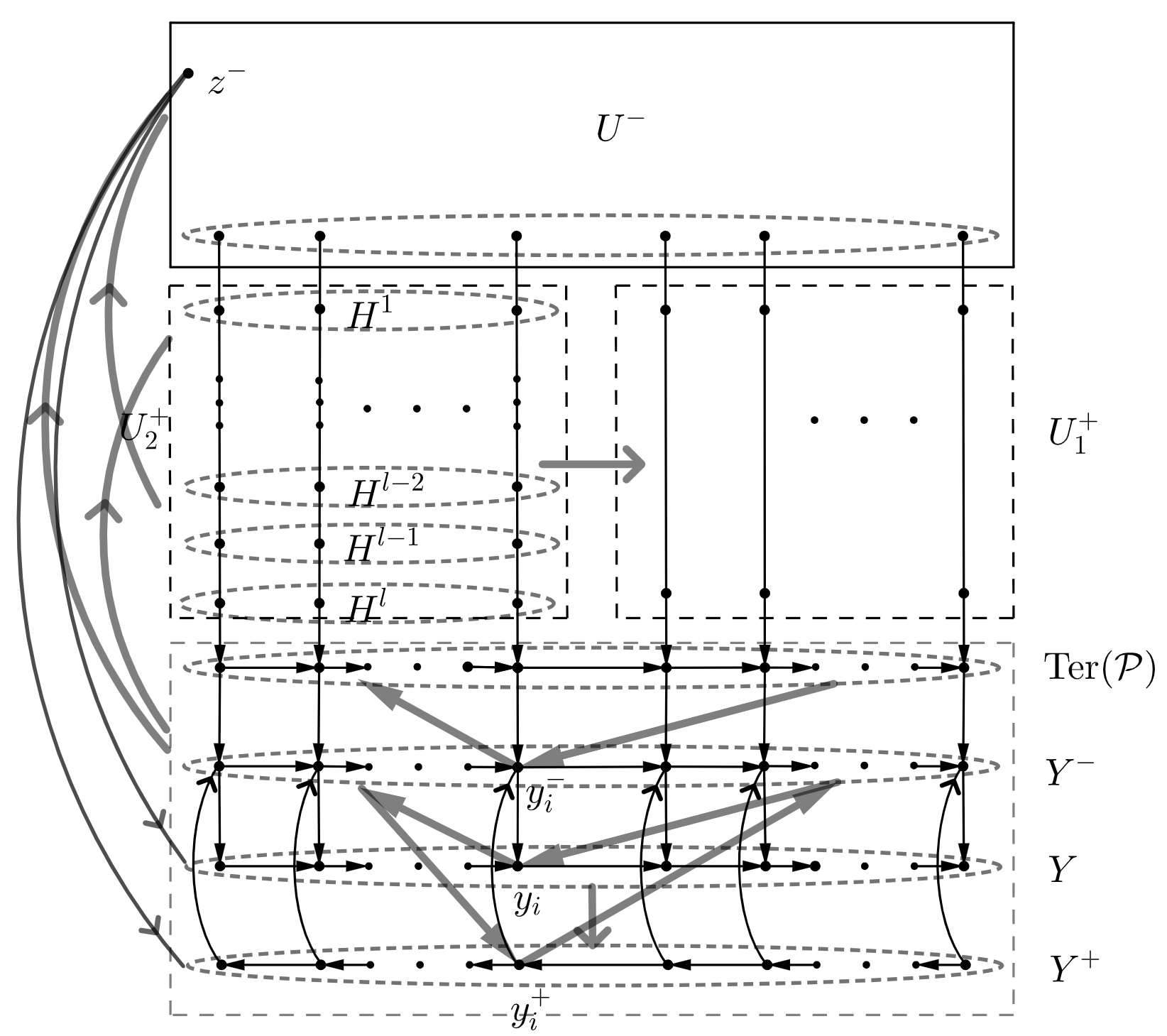}   %以pic.jpg的0.5倍大小输出
\caption{The structure of tournament \(T_1\). }
\label{fig2}
\end{figure}

 To ensure high connectivity, we add arcs that meet the following requirements (as shown in Fig. \ref{fig2}).
\begin{itemize}[itemsep=0pt, topsep=3pt,parsep=2pt]
\item[(B1)] $V(G)\cup U^-$ induces a tournament, where all arcs not in  $\mathcal{P}$ are from \(V(G) \setminus \text{Init}(\mathcal{P})$ to $ U^- \).
    \item[(B2)] The arc set between $Y^-$ and $\text{Ter}(\mathcal{P})$ is \( \{z^{l+1}_j y_i^- \mid j \geq i, \ i\in [k]\} \cup \{y_i^- z^{l+1}_j \mid j < i, \ i\in [k]\} \). Let \( E_1 = \{z^{l+1}_i y_i^- \mid i \in [k]\} \) be the perfect matching from $ \text{Ter}(\mathcal{P})$ to $ Y^-$.
    \item[(B3)] The arc set  between $Y^-$ and $Y$ is \( \{y_j^- y_i \mid j \geq i,\  i\in [k]\} \cup \{y_i y_j^- \mid j < i,\ i\in [k]\} \). Let \( E_2 = \{y_i^- y_i \mid i \in [k]\} \) be the perfect matching from $Y^-$ to $ Y$.
    \item[(B4)] The arc set between $Y^-$ and $Y^+$ is \( \{y_j^- y_i^+ \mid j < i,\ i\in [k]\}\cup \{y_i^+ y_j^- \mid j \geq i,\ i\in [k]\} \) (the direction of arcs in (B4) is opposite to those in (B2) and (B3)).
    \item[(B5)] Choose \( z^- \) to be a vertex in \( U^-\setminus \text{Init}(\mathcal{P}) \), then let \( \{z^-\} \Rightarrow Y\cup Y^+ \).
    \item[(B6)] \( Y \Rightarrow Y^+ \), \( Y \cup Y^+ \Rightarrow V(\mathcal{P}) \cup (U^- \setminus \{z^-\})\), and \( Y^- \Rightarrow \text{Int}(\mathcal{P}) \cup U^- \).
\end{itemize}

It is evident that \( T_1 \) is a tournament. We construct the tournament \( T = (V(T), A(T)) \) based on \( T_1 \), where \( V(T) = V(T_1)\cup  X \cup \{z^+\} \) and \( X = \{x_1, x_2, \ldots, x_k\} \). Partition \( X \) into \( X_1 = \{x_1, \ldots, x_{\lfloor k/2 \rfloor}\} \) and \( X_2 = \{x_{\lfloor k/2 \rfloor + 1}, \ldots, x_k\} \). Then define the arcs to ensure the following conditions hold (as shown in Fig. \ref{fig3}).
\begin{itemize}[itemsep=0pt, topsep=3pt,parsep=2pt]
\item[(C1)] For each \( x_i \in X_1 \), \( N^+_{T \setminus (X \cup Y)}(x_i )= U_1^+ \); for each \( x_i \in X_2 \), $N^+_{T \setminus (X \cup Y)}(x_i )= U_2^+ $. This implies that \( V(T) \setminus (X\cup Y \cup U_1^+) \Rightarrow X_1 \) and \( V(T) \setminus (X\cup Y \cup U_2^+) \Rightarrow X_2 \).  Also $T\langle X \rangle$ is a tournament.
    \item[(C2)] All arcs are directed from \( X \) to \( Y \) except for the arcs in \(\{ y_i x_i\mid \ i \in [k]\} \).
    \item[(C3)] \( U^- \Rightarrow \{z^+\} \Rightarrow V(T) \setminus (U^- \cup \{z^+\})\).
\end{itemize}

The following property can easily be deduced from the construction of $T$.

\begin{figure}[H]
\centering    %居中       %子图居中
   \includegraphics[scale=0.42]{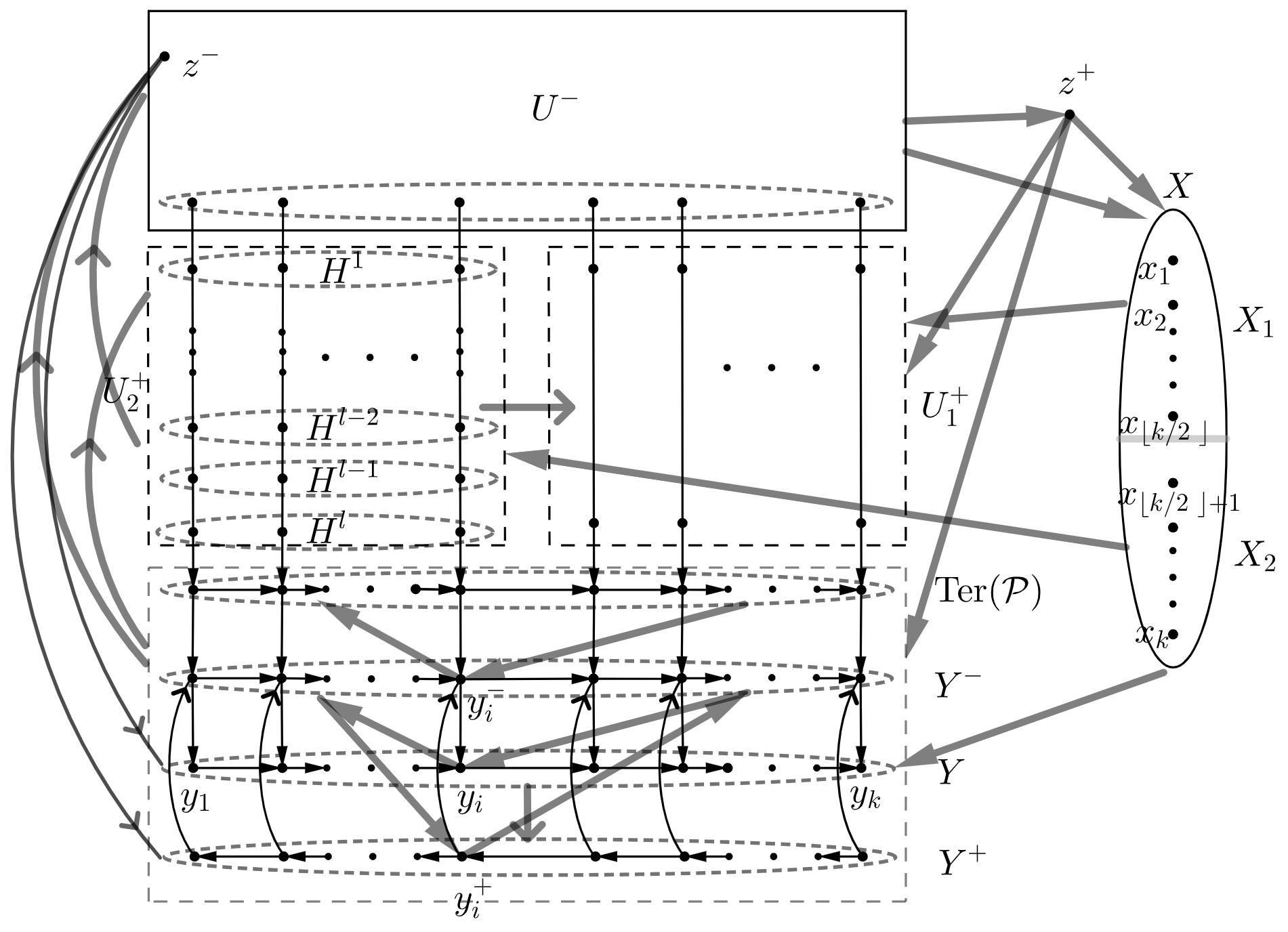}   %以pic.jpg的0.5倍大小输出
\caption{The structure of tournament \(T\). }
\label{fig3}
\end{figure}
\begin{property}\label{prop2}
There exists a collection $\mathcal{P}'$ of $k + 1$ disjoint paths from $U^-$ to $Y \cup \{z^+\}$ in the subdigraph $T \setminus (V(\mathcal{P}) \cup Y^-\cup \{z^-\})$.
\end{property}
\begin{proof}
This is verified by $U^- \Rightarrow X \cup \{z^+\}$ and the fact that there is a perfect matching from $X$ to $Y$ by (C2).
\end{proof}

Now, we verify that \( T \) is a counterexample by showing Claims \ref{cla1}-\ref{cla3}.
\begin{claim}\label{cla1}
\( T \) satisfies \( \delta^+(T) \geq  \frac{k^2+11k}{26}\).
\end{claim}

\begin{proof}
To this end, we analyze the out-degrees of vertices in \(U^-\), \(X\), \( \{z^+\}$ and \(Y\cup Y^+\cup Y^-\cup(V(\mathcal{P})\setminus\text{Init}(\mathcal{P}))\), respectively. Since \( U^- \) induces a regular tournament of order at least $k^2/2$, each vertex in \( U^- \) has out-degree at least \(  \frac{|U^-|-1}{2}\geq \frac{k^2-2}{4} \geq \frac{k^2+11k}{26}  \) (as $k\geq 3$). By (C1)-(C2), each vertex \( x_i \in X \) has out-degree at least 
$$ \min \{ |U_1^+|, |U_2^+|\} +|Y\setminus \{y_i\}| \geq \frac{k^2-15k+26}{26}+k-1=\frac{k^2+11k}{26}.$$ 
Condition (C3) implies that $\{z^+\}\Rightarrow U_1^+\cup U_2^+$, hence $z^+$ has out-degree at least \( \frac{k^2-15k+26}{13}\geq \frac{k^2+11k}{26} \) (as $k\geq 41$). For each vertex \(v\in Y\cup Y^+\cup Y^-\cup V(\mathcal{P})\setminus\text{Init}(\mathcal{P})\), by (B1) and (B6), $v$ has at least $|U^-|-1$ out-neighbours in $U^-$, which is at least \(k^2/2-1\geq \frac{k^2+11k}{26}\). Thus  \(\delta^+(T)\geq \frac{k^2+11k}{26} \), establishing the claim.
\end{proof}

\begin{claim}\label{cla2}
\( T \) is \( (2k+1) \)-connected.
\end{claim}
\begin{proof}
Since \( T\langle U^- \rangle \) is a regular tournament of order at least \( k^2/2 \), Fact~\ref{fac1} implies it is \( \lfloor k^2/3 \rfloor \)-connected. As \( \lfloor k^2/3 \rfloor \geq 2k + 1 \) for \( k \geq 7 \), \( T\langle U^- \rangle \) is \( (2k + 1) \)-connected. To establish that \( T \) is \( (2k + 1) \)-connected, it suffices to show that for any vertex \( v \in V(T) \setminus U^- \), there are \( 2k + 1 \) internally disjoint $(v, U^- )$-paths intersecting only at \( v \), and \( 2k + 1 \) internally disjoint \(( U^-, v )\)-paths intersecting only at \( v \). 

By (B1) and (B6), for each vertex \( v \in (V(\mathcal{P}) \setminus \mathrm{Init}(\mathcal{P})) \cup Y^- \cup Y \cup Y^+ \), there are \(2k+1\) internally disjoint (\(v\), \(U^-\))-paths intersecting only at \(v\). Furthermore, (C1)-(C2) guarantee that each \(x_i \in X\) has \(\frac{k^2 + 11k}{26}\) out-neighbours in \(\mathrm{Int}(\mathcal{P}) \cup Y\). Combined with (B1) and $Y\Rightarrow U^-$, this yields \(\frac{k^2 + 11k}{26} \geq 2k+1\) (valid for \(k \geq 42\)) internally disjoint (\(x_i\), \(U^-\))-paths intersecting only at \(x_i\). Since \( \{z^+\} \Rightarrow Y^- \cup Y\cup Y^+ \Rightarrow U^-\), there have \(2k+1\) internally disjoint (\(z^+\), \(U^-\))-paths intersecting only at \(z^+\). 

Next, we claim that for any vertex \(v \in V(T)\setminus U^-\), there are \(2k+1\) internally disjoint (\(U^-\), \(v\))-paths that only intersect at \(v\) in $T$; for brevity, we subsequently denote this property as \((U^-, v, 2k + 1)\) for the vertex \(v\). The fact \( U^- \Rightarrow X \cup \{z^+\} \) directly establishes the property \( (U^-, v, 2k + 1) \) for all \( v \in X \cup \{z^+\} \). For any vertex \(v \in \text{Int}(\mathcal{P})\), there exist \(k\) disjoint \((\text{Init}(\mathcal{P}), v)\)-paths in \(T\langle V(G) \cup Y^- \rangle\) because there are \(k\) disjoint paths from \(\text{Init}(\mathcal{P})\) to \(Y^-\) through \(\mathcal{P} \circ E_1\) and \(Y^-\Rightarrow \text{Int}(\mathcal{P})\). Then according to Property \ref{prop2} and \( Y \cup \{z^+\} \Rightarrow \text{Int}(\mathcal{P}) \), the property \( (U^-, v, 2k + 1) \) holds for any vertex \( v \in \text{Int}(\mathcal{P}) \). For each terminal vertex \( z^{l+1}_j \in \mathrm{Ter}(\mathcal{P}) \), observe that \( z^{l+1}_i \to z^{l+1}_j \) for \( i < j \) (by the transitivity of \( \text{Ter}(\mathcal{P}) \)) and \( y^-_i \to z^{l+1}_j \) for \( j < i \) (by (B2)). This yields \( k \) desired paths from \( \mathrm{Init}(\mathcal{P}) \) to \( z^{l+1}_j \) through \( \mathcal{P} \circ E_1 \). Also, by Property \ref{prop2} and \( Y \cup \{z^+\} \Rightarrow \mathrm{Ter}(\mathcal{P}) \), the property \( (U^-, z^{l+1}_j, 2k + 1) \) is established.

We now proceed to verify the property \( (U^-, z^{l+1}_j, 2k + 1) \) for each vertex in \( Y^- \), \( Y \), and \( Y^+ \), respectively. The structure of \( T \) ensures \( k \) paths from \( \text{Init}(\mathcal{P}) \) to \( Y^- \) via \( \mathcal{P} \circ E_1 \). Specifically, for each vertex \( y_i^- \in Y^- \),  \( z_j^{l+1} \to y_i^- \) for \( j \geq i \) (by (B2)) provides \( k - i + 1 \) internally disjoint $(\text{Init}(\mathcal{P}), y_i^-)$-paths, while \( y_j^- \to y_i^- \) for \( j < i \) (by the transitivity of \( Y^- \)) supplies \( i - 1 \) desired paths. Thus, combining these paths yields the required \( k \) internally disjoint \((U^-, y_i^-)\)-paths. Additionally, Property \ref{prop2} guarantees a collection \(\mathcal{P}'\) of \( k + 1 \) paths from \( U^-  \) to \( Y \cup \{z^+\} \) in \( T \setminus (V(\mathcal{P}) \cup Y^-\cup \{z^-\}) \). From (B3), the arcs \( \{y_j \to y_i^- \mid i < j\} \) contribute \( k - i \) desired paths using vertices \( \{y_j \mid i < j\} \) to \( y_i^- \). By (B4) and (B6), we have \( \{y_j, i \geq j\} \Rightarrow \{y_j^+, i \geq j\} \Rightarrow y_i^- \), from which we can derive \( i \) desired paths through \( \{y_j^+ \mid i \geq j\} \) to \( y_i^- \). Furthermore, there is a path from a vertex in \( U^- \) to \( z^+ \) via \(\mathcal{P}'\), and since \( z^+ \rightarrow y_i^- \), this provides another path from \( U^- \) to \( y_i^- \). Combining these paths with the previous ones establishes the property \( (U^-, y_i^-, 2k + 1) \).

 For each vertex \( y_i \in Y \), transitivity of \( Y \) gives \( y_j \to y_i \) for all \( j < i \), and (B3) provides \( y_j^- \to y_i \) for \( j \geq i \). Following these paths in \( \mathcal{P} \circ E_1 \circ E_2 \), we obtain \( k \) such paths from \( \mathrm{Init}(\mathcal{P}) \) to \( y_i \). Additionally, \( U^- \Rightarrow (X \setminus \{x_i\}) \cup \{z^+\} \Rightarrow y_i \) ensures \( k \) paths from \( U^- \setminus V(\mathcal{P}) \) to \( y_i \). Together with (B5), this establishes \( (U^-, y_i, 2k + 1) \).

Finally, for each vertex \( y_i^+ \in Y^+ \), the transitivity of \( Y^+ \) ensures that \( y_j^+ \to y_i^+ \) for all \( j > i \). Additionally, (B4) guarantees that \( y_j^- \to y_i^+ \) for all \( j < i \). Moreover, (B4) provides a perfect matching from \( \{ y_j^- \mid j \in [i, k-1] \} \) to \( \{ y_j^+ \mid j \in [i+1, k] \} \). By following the paths in \( \mathcal{P} \circ E_1 \), we obtain \( k - 1 \) paths from \( \mathrm{Init}(\mathcal{P}) \) to \( y_i^+ \). Combined with (B5), Property \ref{prop2}, and considering that \( Y \cup \{ z^+ \} \Rightarrow Y^+ \), we establish that \( y_i^+ \) satisfies the property \( (U^-, y_i^+, 2k + 1) \). In summary, the property \( (U^-, v, 2k + 1) \) holds for every vertex \( v \in V(T) \), thereby completing the proof.
\end{proof}

\begin{claim}\label{cla3}
\( T \) contains no \( k \) disjoint \( (x_i, y_i) \)-paths for \( i \in [k] \).
\end{claim}
\begin{proof}
We use contradiction. Assume \( k \) disjoint \((x_i, y_i)\)-paths exist in \( T \), and let \(\mathcal{Q} = \{Q_1, \ldots, Q_k\}\) denote these paths. At most two paths in \(\mathcal{Q}\), labeled \(Q_a\) and \(Q_b\), contain \(z^-\) and \(z^+\). Note that for all \(i \in [k]\), the in-neighborhood of \(y_i\) in \(T \setminus (X \cup Y \cup \{z^+, z^-\})\) lies within \(Y^-\), and the out-neighborhood of \(x_i\) in \(T \setminus (X \cup Y \cup \{z^+, z^-\})\) is within \(U_1^+ \cup U_2^+\). To reach \( Y \), each path in \(\mathcal{Q} \setminus \{Q_a, Q_b\}\) must pass through \(\text{Ter}(\mathcal{P})\) to \( Y^{-} \) and then to \( Y \), as no arcs exist from \(V(\mathcal{P})\setminus \text{Ter}(\mathcal{P})\) to \(Y^{-} \cup Y\) in \(T_1\setminus \{z^-,z^+\}\). This implies 
\begin{equation}
\text{\(\mathcal{Q}\) must cover at least \(k - 2\) vertices in each of \(Y^{-}\) and \(\text{Ter}(\mathcal{P})\).}
\label{2}
\end{equation}

Let \(\mathcal{Q}'\) be the set consisting of the subpaths from the start vertex to the first vertex in \(\text{Ter}(\mathcal{P})\) for each path in \(\mathcal{Q} \setminus \{Q_a, Q_b\}\). A subpath \(Q'_i \in \mathcal{Q}'\) is termed \textbf{acceptable} if it starts at \(x_i\) and ends at \(z^{l+1}_j \in \text{Ter}(\mathcal{P})\) with \(i, j \in I_1\) or \(i, j \in I_2\), where \(I_1 = [\lfloor k/2 \rfloor]\) and \(I_2 = [k]\setminus [\lfloor k/2 \rfloor]\).

We assert that there are at most 4 non-acceptable subpaths starting from \(X_2\) among \(\mathcal{Q}'\). For any such non-acceptable subpath, say an \( (x_i, z^{l+1}_j)\)-path \(Q'_i\), we have \(i \in I_2\) and \(j \in I_1\). According to our construction (mainly because (B2)-(B3)), \(Q_i[z^{l+1}_j,y_i] \) (\(Q_i \) is the path that contains the subpath $Q'_i$ in $\mathcal{Q}$) has to utilize at least two vertices in \(\text{Ter}(\mathcal{P})\) or \(Y^-\). Combined with (\ref{2}), this limits the number of non-acceptable subpaths from \(X_2\) to at most 4. Hence, \(\mathcal{Q}'\) contains at least \(\lceil k/2 \rceil - 4 - x\) acceptable subpaths starting from \(X_2\), where \(x\) denotes the number of paths in \(\{Q_a, Q_b\}\) that start from \(X_2\). This occupies at least \(\lceil k/2 \rceil - 4 - x\) vertices in $\{z^{l+1}_i\mid i\in I_2\}$. Consequently, there are at most \(4 + x\) non-acceptable subpaths starting from \(X_1\), as each such subpath must use a vertex in \(\{z^{l+1}_i \mid i \in I_2\}\). Note that \(\mathcal{Q}'\) has at least $(\lfloor k/2 \rfloor -2+x)$ subpaths starting from \(X_1\). This yields that the number of acceptable subpaths starting from \(X_1\) in \(\mathcal{Q}'\) is at least \((\lfloor k/2 \rfloor -2+x) - (4+ x)=\lfloor k/2 \rfloor -6\).

Let \(X^+\) denote the set of second vertices from all paths in \(\mathcal{Q}\). Define \(X^+_1 = X^+ \cap U^+_1\) and \(X^+_2 = X^+ \cap U^+_2\), where \(|X^+_2| = \lceil k/2 \rceil\). Since \(U_2^+ \Rightarrow U_1^+\) and there are no arcs from \(U^- \cup (V(G) \setminus \text{Ter}(\mathcal{P}))\) to \(Y^- \cup Y \cup Y^+\) in \(T \setminus (X\cup \{z^-, z^+\})\), the \(\lfloor k/2 \rfloor - 6\) acceptable subpaths from \(X_1\) must consist of \(\lfloor k/2 \rfloor - 6\) paths from \(X_1\) to \(X^+_1\), \(\lfloor k/2 \rfloor - 6\) paths from \(X^+_1\) to \(H^0\), and \(\lfloor k/2 \rfloor - 6\) paths from \(H^0\) to \(H^{l+1}\) in \(G \setminus \{z^{\lfloor k/2 \rfloor +1}_i, \ldots, z^{l+1}_k\}\), while avoiding the \(\lceil k/2 \rceil\) vertices in \(X^+_2\). This implies there are \(\lfloor k/2 \rfloor - 6\) paths from \(H^0\) to \(H^{l+1}\) in \(G\setminus \{z^{\lfloor k/2 \rfloor +1}_i, \ldots, z^{l+1}_k\}\) that avoid \(X^+_2\). However, this contradicts Property \ref{prop1}.
\end{proof}
\end{proof}

\section{Proof Theorem \ref{main2}}

\subsection{Preliminary}
In our proofs, we utilize the following easy consequence of Menger's Theorem.
\begin{theorem}\label{menger}
\cite{menger} Let $D$ be a $k$-connected digraph. Then for any two disjoint sets $\{x_1, \ldots$, $x_k\}$ and $\{y_1, \ldots, y_k\}$ of $V(D)$, there exist $k$ disjoint $(x_i, y_{\pi(i)})$-paths for some permutation $\pi$ of $\{1, 2, \ldots, k\}$.
\end{theorem}

%\jbj{The following definitions and Lemma \ref{lemma8} plays a central role in our proofs.}

We define the main tools, namely, ``nearly out-dominating vertex" and ``nearly in-dominating set". Let \( c \in \mathbb{N} \), and let \( u, v \) be two vertices of a digraph \( D \). We say $v$ is \textbf{$\boldsymbol{c}$-in-good for $\boldsymbol{u}$} in $D$ if either $v\rightarrow u$ or at least \( c \) internally disjoint \( (v, u) \)-paths of length 2 exist in \( D \). And we say $v$ is \textbf{$\boldsymbol{c}$-out-good for $\boldsymbol{u}$} in $D$ if either $u\rightarrow v$ or there exist at least $c$ internally disjoint $(u, v)$-paths of length 2 in $D$. Similar definition (i.e., $c$-good for $u$) have appeared in \cite{Zhou(2025)}. However, the newly introduced strategy of adjustment paths requires the definition of a ``nearly out-dominating vertex". Thus, we further distinguish between \( c \)-in-good for \( u \) and \( c \)-out-good for \( u \).

\begin{definition}\label{def1}
 A vertex $u$ is a \textbf{nearly in-dominating  vertex} of $D$ if, for every $c \in \mathbb{N}$, all but at most $2c$ vertices are $c$-in-good for $u$. Similarly, a vertex $u$ is a \textbf{nearly out-dominating vertex} of $D$ if all but at most $2c$ vertices are $c$-out-good for $u$ for every $c \in \mathbb{N}$.
\end{definition} 
 \begin{definition}\label{def2}
  We call a set $U$ of vertices  a \textbf{nearly in-dominating  set} of $D$ if, for every vertex $u\in U$ and every $c \in \mathbb{N}$, all but at most $2c$ vertices in $D \setminus U$ are $c$-in-good for $u$ in $D$.
\end{definition}

We need the following simple lemma in \cite{Zhou(2025)}, which establishes the existence of a nearly in-dominating vertex and a nearly out-dominating vertex in a semicomplete digraph. For readers' convenience, we include a short proof here.
% from \cite{Zhou(2025)}We need the following simple lemma from \cite{Zhou(2025)}, establishes the existence of a nearly in-dominating vertex in a semicomplete digraph. For the convenience of readers, we append this short proof here. in the appendixIndeed, the existence of a nearly out-dominating vertex in a semicomplete digraph can be proved symmetrically. 
\begin{lemma}\label{key1}
\cite{Zhou(2025)} Every semicomplete digraph $D$ contains a nearly out-dominating  vertex. Symmetrically, every semicomplete digraph $D$ contains a nearly in-dominating  vertex.
\end{lemma}
\begin{proof}
Fix \( c \in \mathbb{N} \) as in the definition of nearly out-dominating vertex, and let \( T \) be a spanning tournament of \( D \). Choose \( u \) to be a vertex of maximum out-degree in \( T \). Define  
$$
S = \{ w \in N^-_T(u) \mid w \text{ is not \( c \)-out-good for \( u \) in } D \}.
$$ 
If $|S|<2c$, by defibition, $u$ is a nearly out-dominating vertex. Hence, assume \( |S| > 2c \). By the definition of \( c \)-out-goodness, each \( w \in S \) lacks \( c \) independent \((u, w)\)-paths of length 2 in \( T \), implying  $|N^-_T(w) \cap N^+_T(u)| \leq c - 1.$ Combining this with \( d^+_T(w) \leq d^+_T(u) \), we derive that   
\begin{equation*}
\begin{aligned}
|N^-_{T}(w) \cup N^+_{T}(u)|
&= |N^-_{T}(w)|+| N^+_{T}(u)|-|N^-_{T}(w) \cap N^+_{T}(u)|\\
&\geq |T|-1-d^+_{T}(w)+d^+_{T}(u)-(c-1)\\
&\geq |T|-1-c+1\\
&\geq |T|-c.
\end{aligned}
\end{equation*}
Thus, at most \( c \) vertices in \( T \) are neither in-neighbours of \( w \) nor out-neighbours of \( u \). Since \( S \Rightarrow u \), every \( w \in S \) has at least \( |S| - c \) in-neighbours within \( S \), forcing \( \delta^-(T\langle S \rangle) \geq |S| - c \). However, in any tournament, \( \delta^-(T\langle S \rangle) \leq \frac{|S| - 1}{2} \). So \(|S| - c \leq \frac{|S| - 1}{2}\), from which we can derive \(|S| \leq 2c - 1\). This contradicts the previously stated \(|S| > 2c\), and thus \(u\) is a nearly out-dominating vertex in \(D\). By reversing all arcs of $D$, we obtain that every semicomplete digraph $D$ contains a nearly in-dominating  vertex.
\end{proof}
%An \textbf{in-king} in a tournament $T$ is a vertex $v$ such that for every other vertex $u\in V(T)$, there is a path from $u$ to $v$ of length at most 2. Landau \cite{Landau(1953)} proved that the vertex with the maximum in-degree in every tournament is an in-king. Actually, Lemma \ref{lemma8} shows that this in-king is also a nearly in-dominating vertex, thereby {significantly strengthening the properties of those in-kings which are also vertices of maximum in-degree}.

The following definition will be repeatedly used in the proof of Theorem \ref{main2}.

\begin{definition}\label{def3}
Let $\gamma$ be a positive integer, and $D$ be a digraph with a vertex $v$ and a subset $U$ of $V(D)\setminus \{u\}$, we say that the vertex $v$ is:
\begin{itemize}[itemsep=0pt, topsep=3pt,parsep=2pt]
    \item a \textbf{$\boldsymbol{\mathbf{\gamma}}$-out-dominator of $\boldsymbol{U}$} if $v$ has at least $\gamma$ out-neighbours in $U$; and
    \item a \textbf{$\boldsymbol{\gamma}$-in-dominator of $\boldsymbol{U}$} if $v$ has at least $\gamma$ in-neighbours in $U$.
\end{itemize}
\end{definition}

\subsection{Proof of Theorem \ref{main2}}
Suppose \( D \) is a \((2k+1)\)-connected semicomplete digraph with \( \delta^+(D) \geq 7k^2 + 36k \) and disjoint subsets \( X = \{x_1, \ldots, x_k\} \), \( Y = \{y_1, \ldots, y_k\} \). Our purpose is to link $x_i$ to $y_i$ for each $i\in [k]$. Through iterative application of Lemma \ref{key1} to \( D \setminus (X \cup Y \cup \bigcup_{j < i} \{u_j\}) \), we construct a set \( U = \{u_1, \ldots, u_{3k}\} \), where each \( u_i \) is a nearly in-dominating vertex within \( D \setminus (X \cup Y \cup \bigcup_{j < i} \{u_j\}) \). For each \( x \in X \), define \( N_x \) as the collection of all \((2k+1)\)-out-dominators of \( U \) within \( N^+_{D \setminus (X \cup Y \cup U)}(x) \). According to the order of $N_x$, partition \( X \) into \( X_1 \) and \( X_2 \), where \( X_1 = \{x \in X \mid |N_x| \geq 7k^2 + 6k + 1\} \), and \( X_2 = X \setminus X_1 \), and note that \( X_1 \) or \( X_2 \) may be empty. Let \( N = \bigcup_{x \in X_1} N_x \); the vertices in \( N \) play a crucial role in finding disjoint short paths from \( X_1 \) to \( U \). If \( N \neq \emptyset \), apply Lemma \ref{key1} to select a nearly out-dominating vertex \( y_{k+1} \in N \) within the subgraph \( D\langle N \rangle \); otherwise, \( N = \emptyset \), i.e., $X_1=\emptyset$, the vertex \( y_{k+1} \) is unnecessary. To finalize the proof of Theorem \ref{main2}, three subsequent steps are necessary. Below, we present an outline of these steps.

\textbf{Sketch of proof.}  
Initially, by $(2k+1)$-connectivity of $D$, Menger's theorem yields a collection \(\mathcal{Q}\) of \(k+1\) disjoint paths from \(U\) to \(Y \cup \{y_{k+1}\}\) in \(D \setminus X\). Next, our aim is to find a collection \(\mathcal{Q}^{**}\) of \(k\) disjoint paths from \(U\) to \(Y\) that minimizes \(|V(\mathcal{Q}^{**})|\), and construct a set \(\mathcal{P} = \{P_i\}_{i \in [k]}\) of \(k\) disjoint paths from \(X\) to \(X \cup U\), each of length at most 3. Additionally, \(V(\mathcal{P}) \cap V(\mathcal{Q}^{**}) = \emptyset\), and each vertex in \(\text{Ter}(\mathcal{P})\) is a 1-in-dominator of \(U \setminus \text{Init}(\mathcal{Q}^{**})\). To address the challenge of the existence of such paths starting from \(X_1\) in \(\mathcal{P}\), we introduce a path-adjustment program. Initiated from \(\mathcal{Q}\), the program leads to the optimized path system \(\mathcal{Q}^*\), and then \( \mathcal{Q}^{**} \) is derived by modifying \( \mathcal{Q}^* \). Finally, by using the minimality of \(|V(\mathcal{Q}^{**})|\) and the properties of \(\text{Ter}(\mathcal{P})\),  we obtain \(k\) disjoint paths to link \(\text{Ter}(\mathcal{P})\) with \(\text{Init}(\mathcal{Q}^{**})\) and these paths are internally disjoint from \(V(\mathcal{P})\cup V(\mathcal{Q}^{**})\). 
%\hfill $\square$

\textbf{Proof of Theorem \ref{main2}.} Since $ D \setminus X $ is $ (k+1) $-connected (by deleting $ |X| = k $ vertices from a $ (2k+1) $-connected digraph), Menger's Theorem guarantees $ k+1 $ disjoint paths from $ U $ to $ Y \cup \{y_{k+1}\} $ in $ D \setminus X $. Notice that if \(X_1 = \emptyset\), we choose no vertex \(y_{k+1}\) and directly find \( k \) disjoint paths from \( U \) to \( Y \) in \( D \setminus X \). Indeed, we assume \(X_1 \neq \emptyset\), as the case \( X_1 = \emptyset \) can be directly inferred from the arguments for \( X_1 \neq \emptyset \), with a simpler proof. Let $ \mathcal{Q} = \{Q_1, \ldots, Q_{k+1}\} $ denote such a collection, where each $ Q_i $ is a path from $ q_i \in U $ to $ y_i \in Y \cup \{y_{k+1}\} $. We select $ \mathcal{Q} $ to minimize the total number of vertices $ |V(\mathcal{Q})| $, ensuring each $ Q_i $ is minimal and $V(\mathcal{Q})\cap U=\text{Init}(\mathcal{Q}) $. 

Next, we aim to construct
\begin{itemize}[itemsep=0pt, topsep=3pt, parsep=2pt]
    \item[(I)] a path system $ \mathcal{Q}^* $ from $ U $ to $ Y \cup \{v\} $ (for some $ v \in N $), where each path is minimal in $ D \setminus X $ (i.e.,  $ V(\mathcal{Q}^*) \subseteq V(D) \setminus X $);
    \item[(II)] a subset $ X_1^* \subseteq N $ with $ |X_1^*| = |X_1| $, such that $ X_1^* \cap V(\mathcal{Q}^*) = \emptyset $, and there exists a perfect matching $ f: X_1 \to X_1^* $ in $ D $ (i.e., each $ (x_i, f(x_i)) $ is an arc).
\end{itemize}
 For such a path system \(\mathcal{Q}^*\), the terminal vertex of the path in \(\mathcal{Q}^*\) that is not in \(Y\) is called the \textbf{special terminal vertex} of \(\mathcal{Q}^*\). Initially, set \(\mathcal{Q}^* = \mathcal{Q}\), \(X_1^* = \emptyset\), and \(W = \emptyset\) (the set of vertices that are not candidates for the special terminal vertex of the new path system \(\mathcal{Q}^*\)). Let $I\subseteq X_1$ be the set such that there is a perfect matching from $I$ to $X^*_1$ during this process. Thus, in the initial state \(I = \emptyset\). To obtain the desired \(\mathcal{Q}^*\) and \(X_1^*\), we introduce a path-adjustment program, which consists of two procedures that are iteratively used.%Without loss of generality, assume that $X_1 = \{x_1, \ldots, x_l\}$. 

\textbf{The procedure (A1)}: If there exists a vertex \(x_i \in X_1\setminus I\) such that \(N_{x_i} \setminus (V(\mathcal{Q}^*) \cup X^*_1) \neq \emptyset\), select a vertex \(x^*_i\) in \(N_{x_i} \setminus (V(\mathcal{Q}^*) \cup X^*_1)\). Update \(I:=I \cup \{x_i\}\) and \(X^*_1:=X^*_1 \cup \{x^*_i\}\). Repeat this process until we obtain sets \(I\) and \(X^*_1\) such that 
\begin{equation}\label{3}
\text{$I=X_1$ or $N_{x_i}\subseteq (V(\mathcal{Q}^*)\cup X^*_1)$ for any $x_i\in X_1\setminus I$.}
\end{equation}
Procedure (A1) ensures that there exist sets \(X^*_1\) and $\mathcal{Q}^*$ satisfying (I)-(II) when \(I = X_1\). When \(I \neq X_1\), we perform the following procedure (A2) once to obtain a new path system \(\mathcal{Q}^*\). This increases the number of vertices in \(I\) and that in \(X^*_1\). We then return to procedure (A1) until procedure (A1) terminates. Unless \(I = X_1\), we execute procedure (A2) once and repeat the process. Therefore, the entire process will definitely terminate, and require at most $|X_1|$ executions of procedure (A2). Specifically, the input \(\mathcal{Q}^*\) to procedure (A2) must have its special terminal vertex (denoted as \(o\)) satisfy the following condition:
\begin{equation}
\begin{split}
&\text{ all but at most } 4k+4 \text{ vertices in } \bigcup_{x_i\in X_1\setminus I}N_{x_i}\setminus (W\cup X_1^*) \\
&\text{\ \ \ \ \ \ \ \ \ \ \   are } (2k+2)\text{-out-good for $o$ in $D\langle N \rangle$.}
\end{split}
\label{7}
\end{equation}
This condition is verified at the initial step due to the selection of \(y_{k+1}\) and Definition \(\ref{def1}\) with \(c = 2k + 2\). The iterative step is verified after each completion of procedure (A2), as procedure (A1) does not modify the path system \(\mathcal{Q}^*\).

Now, we introduce the \textbf{procedure (A2)}: Initially, set \( I \), \( X^*_1 \), and \(\mathcal{Q}^*\) to the final outputs of the previous procedure (A1), and let \( o \) be the special terminal vertex of \(\mathcal{Q}^*\). Next, sequentially update \( W \), \( o \) (denoted as \( o' \) to distinguish it from the original vertex), \(\mathcal{Q}^*\), \( X^*_1 \), and \( I \). Define \( O_{good} \) as the set of all vertices in \(\bigcup_{x_i\in X_1\setminus I}N_{x_i}\setminus (W\cup X_1^*)\) that are \((2k+2)\)-out-good for \( o \) in \( D\langle N \rangle \) (these vertices can be used to adjust the paths in \(\mathcal{Q}^*\)). Append all vertices in \(\bigcup_{x_i\in X_1\setminus I}N_{x_i}\setminus (W\cup X_1^*)\) that are not \((2k+2)\)-out-good for \( o \) in \( D\langle N \rangle \) into \( W \). By (\ref{7}), there are at most \(4k+4\) such vertices. Then, put all vertices in \( X_1^{*} \) into $W$. To find a suitable vertex to add to \( X_1^{*} \), let \( F \) be the final two vertices of \( O_{good} \cap Q \) for each \( Q \in \mathcal{Q}^* \) (if they exist), and then append \( F \) into \( W \). Then \( W \) increases by at most \(4k+4+|X_1^*|+|F| \leq 7k+6\) vertices compared to its previous value.

Let \( O_{cand} = \bigcup_{x_i\in X_1\setminus I}N_{x_i} \setminus W \). Since \( X_1^{*} \subseteq W \), we have \( O_{cand} = \bigcup_{x_i\in X_1\setminus I}N_{x_i} \setminus (W \cup X_1^{*}) \). Since the procedure (A2) is executed at most \( |X_1| \leq k \) times in total, and each execution increases \( W \) by at most \( 7k + 6 \) vertices, the set \( O_{cand} \) remains non-empty at each iteration.  This is verified by \( |N_{x_i}| \geq 7k^2 + 6k + 1 \) for all \( i \in [l] \) and the following calculation:
\[
\left|\bigcup_{x_i\in X_1\setminus I} N_{x_i} \setminus W\right| \geq (7k^2 + 6k + 1) - (7k + 6)k \geq 1.
\]
By applying Lemma \ref{key1}, we can choose a vertex \( o' \in O_{cand} \) to be a nearly out-dominating vertex in \( D\langle O_{cand} \rangle \).

Since (A1) terminates with \( N_{x_i} \subseteq V(\mathcal{Q}^*) \cup X^*_1 \) for any \( x_i \in X_1 \setminus I \), and \( o' \notin W \), we have \( o' \in V(Q_j) \) for some \( Q_j \in \mathcal{Q}^* \) and \( o' \notin F \). Thus, there are at least two vertices, say \( x^* \) and \( x^{**} \), in \( O_{good} \) after \( o' \) along the path \( Q_j \), with \( x^* \) appearing before \( x^{**} \). Assume \( x^{*} \in N_{x_r} \) for some \( x_r \in X_1 \setminus I \). Our idea is to update \(\mathcal{Q}^*\), release the vertex \( x^{*} \), add it to \( X^*_1 \), and thus enlarge \( X^*_1 \) and $I$. To this end, we update \(\mathcal{Q}^*\) as follows. Since \( x^{**} \in O_{good} \), either \( o \rightarrow x^{**} \) or there exist at least \( 2k+2 \) internally disjoint 2-paths from \( o \) to \( x^{**} \). In the former case, update \( Q_j := Q_{k+1}[q_{k+1}, o] \circ Q_j[x^{**}, y_j] \) (see Fig. \ref{fig:multiple-images} (a)). In the latter case, there exist at least \( k+2 \) internally disjoint 2-paths from \( o \) to \( x^{**} \) that are internally disjoint with \( X_1^{*} \). If a 2-path \( owx^{**} \) exists with \( w \notin V(\mathcal{Q}^*) \cup X_1^{*} \), then update \( Q_{j} := Q_{k+1}[q_{k+1}, o] \circ owx^{**} \circ Q_j[x^{**}, y_j] \) (see Fig. \ref{fig:multiple-images} (b)). Otherwise, by the pigeonhole principle, there exists a path \( Q_m \in \mathcal{Q}^* \) containing two vertices \( r_1 \) and \( r_2 \) such that \( o \rightarrow r_1 \rightarrow x^{**} \) and \( o \rightarrow r_2 \rightarrow x^{**} \). Due to the minimality of each path in \(\mathcal{Q}^*\), we have \( m \neq j \). Update \( Q_{m} := Q_{k+1}[q_{k+1}, o] \circ Q_m[r_2, y_m] \) and \( Q_{j} := Q_{m}[q_{m}, r_1] \circ Q_j[x^{**}, y_j] \). If \( m = k+1 \), simply update \( Q_{j}:= Q_{m}[q_{m}, r_1] \circ Q_j[x^{**}, y_j] \) (see Fig. \ref{fig:multiple-images} (c)). Also, in either case, update \( Q_{k+1} := Q_j[q_j, o'] \). Moreover, each path is chosen to be minimal.

% in \(\mathcal{Q}^*\)
Notably, the process has released the vertex \( x^{*} \in O_{\text{good}} \subseteq \bigcup_{x_i\in X_1\setminus I} N_{x_i} \setminus (W \cup X_1^{*}) \).  
Recall that \( x^{*} \in N_{x_r} \) with \( x_r \in X_1 \setminus I \).  
We update \( I:= I \cup \{x_r\} \) and \( X_1^{*}:= X_1^{*} \cup \{x^{*}\} \).  
Recall that \( o' \in O_{\text{cand}} \) is a nearly out-dominating vertex in \( D\langle O_{\text{cand}} \rangle \).  
Furthermore, observe that:  
\[
\bigcup_{x_i \in X_1 \setminus (I \cup \{x_r\})} N_{x_i} \setminus (W \cup X_1^{*}) \subseteq O_{\text{cand}}.
\]  
By Definition \ref{def1}, all but at most \( 4k + 4 \) vertices in \( \bigcup_{i \in X_1 \setminus (I \cup \{x_r\})} N_{x_i} \setminus (W \cup X_1^{*}) \) are \( (2k + 2) \)-out-good for \( o' \) in \( D\langle O_{\text{cand}} \rangle \subseteq D\langle N \rangle \), fulfilling condition (\ref{7}).  
Finally, relabel the vertices in \( \text{Init}(\mathcal{Q}^*) \) such that each \( Q_i \in \mathcal{Q}^* \) remains a \( (q_i, y_i) \)-path for \( i \in [k] \), and \( Q_{k+1} \in \mathcal{Q}^* \) becomes a \( (q_{k+1}, o') \)-path.  
Proceed to run procedure (A1) unless \( I = X_1 \).
  
  \begin{figure}[htbp]
    \centering
    \begin{subfigure}[c]{0.3\textwidth}
        \centering
        \includegraphics[height=5cm, keepaspectratio]{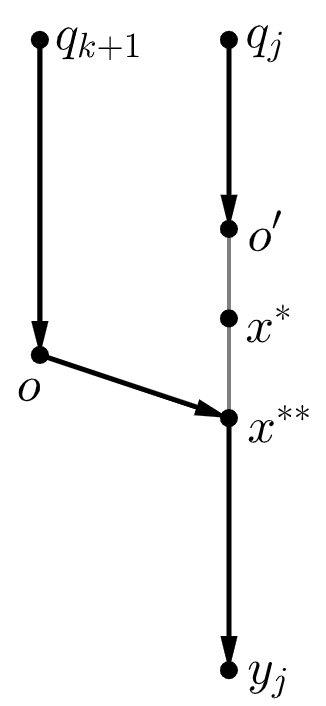}
        \caption{}
    \end{subfigure}
    \hfill
    \begin{subfigure}[c]{0.3\textwidth}
        \centering
        \includegraphics[height=5cm, keepaspectratio]{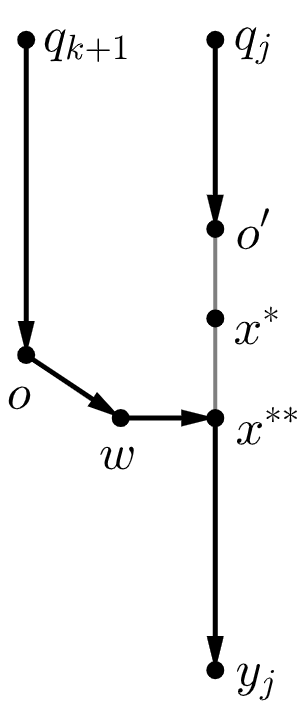}
        \caption{}
    \end{subfigure}
    \hfill
    \begin{subfigure}[c]{0.3\textwidth}
        \centering
        \includegraphics[height=5cm, keepaspectratio]{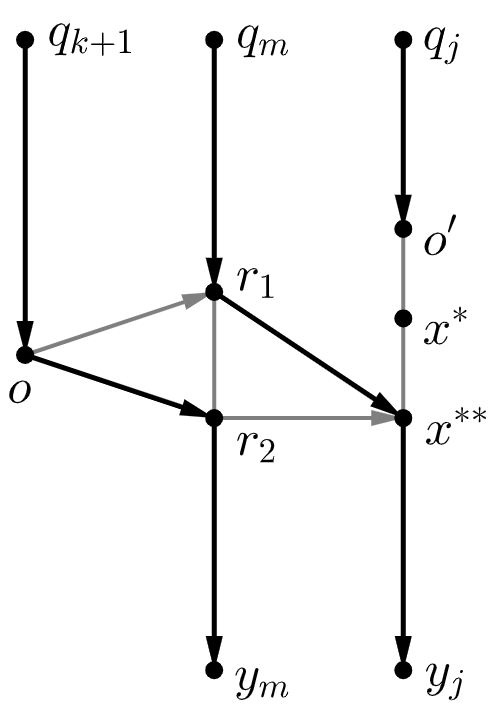}
        \caption{}
    \end{subfigure}
    \caption{Three strategies for adjusting the paths in $\mathcal{Q}^*$.}
    \label{fig:multiple-images}
\end{figure}

The entire path-adjustment program will definitely terminate, that is, it will eventually output a set \(I = X_1\), a set \(X_1^*\) consisting of \(|X_1|\) vertices, a special terminal vertex \(o\), and \(k+1\) disjoint paths \(\mathcal{Q}^*\) from \(U\) to \(Y \cup \{o\}\), such that (I)-(II) hold. In the subsequent proof, the path in \(\mathcal{Q}^*\) that terminates at \(o\) will no longer be needed. Therefore, we define \(\mathcal{Q}^{**}\) as the \(k\) disjoint paths from \(U\) to \(Y\) in \(\mathcal{Q}^*\). To construct subsequent paths, $\mathcal{Q}^{**}$ is chosen such that $|V(\mathcal{Q}^{**})|$ is minimized.

\begin{claim}
There exists a family \( \mathcal{P}_1 = \{P_i\}_{i \in [l]} \) of \( l \) pairwise disjoint \((x_i, p_i)\)-paths of length at most 3 from \( X_1 \) to \( \big(U \setminus \mathrm{Init}(\mathcal{Q}^{**})\big) \) with $p_i\in \big(U \setminus \mathrm{Init}(\mathcal{Q}^{**})\big)$, such that \( V(\mathcal{P}_1) \cap (V(\mathcal{Q}^{**})\cup X_2) = \emptyset \) and  
\begin{equation}
\begin{split}
&\text{\ \ \ \ \ every terminal } p_i \text{ has at least } 25k \text{ out-neighbours in } \\
& D \setminus (X \cup Y \cup U) \text{ that are 1-in-dominators of } U \setminus \mathrm{Init}(\mathcal{Q}^{**}).
\end{split}
\label{8}
\end{equation}
Moreover, let \( P_i \) be a trivial path that has only one vertex \( x_i \) for each \( x_i \in X_2 \), then $P_i$ also satisfies \( V(P_i) \cap V(\mathcal{Q}^{**}) = \emptyset \) and (\ref{8}).
\end{claim}

\begin{proof}
First, we construct the paths \( \mathcal{P}_1 = \{P_i\}_{i \in [l]} \). A perfect matching \( M_1 \) from \( X_1 \) to \( X_1^* \) exists, where \( X_1^* \subseteq N \). Since each \( x_i^* \in X_1^* \) is a \( 2k \)-out-dominator of \( U \), it has at least \( 2k \) out-neighbours in \( U \). As \( |U \cap V(\mathcal{Q}^{**})| = k \), there is a perfect matching \( M_2 \) from \( X_1^* \) to \( U \setminus V(\mathcal{Q}^{**}) \). The composition \( M_1 \circ M_2 \) yields \( l \) disjoint 2-paths from \( X_1 \) to \( U \setminus V(\mathcal{Q}^{**}) \), avoiding \( V(\mathcal{Q}^{**}) \cup V(\mathcal{P}_2) \). For each vertex \( v \in U \setminus V(\mathcal{Q}^{**}) \), the number of out-neighbours in \( D \setminus (X \cup Y \cup U) \) is at least \( \delta^+(D) - |X \cup Y \cup U| \geq 25k \). Each such out-neighbour \( u \) is a 1-in-dominator of \( U \setminus \mathrm{Init}(\mathcal{Q}^{**}) \), as \( v \) is an in-neighbour of \( u \). Thus, \( \mathcal{P}_1 \) satisfies (\ref{8}).  

Now, we verify $P_i$ meets (\ref{8}) for each \( x_i \in X_2 \). By the selection of \( X_2 \), for each \( x_i \in X_2 \), the set \( N^+_D(x_i) \setminus (X \cup Y \cup U) \) contains at most \( 7k^2 + 6k \) distinct \( 2k \)-out-dominators of \( U \). Since \( D \) is semicomplete and \( |U| = 3k \), every vertex in \( N^+_D(x_i) \setminus (X \cup Y \cup U) \) not being a \( 2k \)-out-dominator of \( U \) must be a \( (k+1) \)-in-dominator of \( U \), and is a 1-in-dominator of \( U \setminus \mathrm{Init}(\mathcal{Q}^{**}) \) as \( |\mathrm{Init}(\mathcal{Q}^{**})|=k \). Therefore, the number of 1-in-dominators of \( U\setminus \mathrm{Init}(\mathcal{Q}^{**}) \) in \( N^+_D(x_i) \setminus (X \cup Y \cup U) \) satisfies  
\[
\delta^+(D) - |X \cup Y \cup U| - (7k^2 + 6k) \geq 25k.
\]
The path \( P_i \) ensures \( V(P_i) \cap V(\mathcal{Q}^{**}) = \emptyset \) and fulfills (\ref{8}) as desired.
\end{proof}

Define \( \mathcal{P} = \{P_i\}_{i \in [k]} \). The proof is completed by constructing a set \( \mathcal{R} \) of \( k \) disjoint paths linking \( p_i \) to \( q_i \) for each \( i \in [k] \) in the digraph $(D \setminus (V(\mathcal{Q}^{**}) \cup V(\mathcal{P}))) \cup \{p_1, \ldots, p_k, q_1, \ldots, q_k\}.$ To achieve this, we prove the following claim.

\begin{claim}\label{cla5}
For each \( i \in [k] \), there exist at least \( 2k \) internally disjoint \((p_i, q_i)\)-paths of length at most 3 in  $(D \setminus (V(\mathcal{Q}^{**}) \cup V(\mathcal{P}))) \cup \{p_i, q_i\}.$
\end{claim}

\begin{proof}
For each \( i \in [k] \), define \( N^-_{p_i} \) as the set of all 1-in-dominators of \( U \setminus \mathrm{Init}(\mathcal{Q}^{**}) \) in \( N^+(p_i) \setminus (X \cup Y \cup U) \). By condition (\ref{8}), we have \( |N^-_{p_i}| \geq 25k \). The minimality of \( |V(\mathcal{Q}^{{**}})| \) implies \( |N^-_{p_i} \cap V(\mathcal{Q}^{{**}})| \leq k \). Suppose otherwise, then there exists a vertex \( v \in N^-_{p_i} \cap Q_j \) that is the third or a subsequent vertex on \( Q_j \) (note \( \mathrm{Init}(\mathcal{Q}^{**}) \subseteq U \), so \( \mathrm{Init}(\mathcal{Q}^{**}) \cap N^-_{p_i} = \emptyset \)). Since \( v \) is a 1-in-dominator of \( U \setminus \mathrm{Init}(\mathcal{Q}^{**} ) \), there exists \( u \in U \setminus \mathrm{Init}(\mathcal{Q}^{**}) \) with \( u \to v \). Replacing \( Q_j \) with \( uv \circ Q_j[v, y_j] \) yields a shorter path, contradicting the minimality of \( \mathcal{Q}^{**} \).  Since \( |\mathrm{Int}(\mathcal{P})| \leq 2k \), the remaining set $B = N^-_{p_i} \setminus (V(\mathcal{Q}^{**}) \cup V(\mathcal{P}))$   
satisfies \( |B| \geq 22k \), where every \( v \in B \) is a 1-in-dominator of \( U \setminus \mathrm{Init}(\mathcal{Q}^{**}) \). 

Since \( U \) is nearly in-dominating in \( D \setminus (X \cup Y) \) and \( q_i \in U \), Definition \ref{def2} with \( c = 10k \) implies at most \( 20k \) vertices in \( D \setminus (X \cup Y \cup U) \) are not \( 10k \)-in-good for \( q_i \). Thus, \( B \) contains a subset \( C \) (\( |C| = 2k \)) where each \( v \in C \) is \( 10k \)-in-good for \( q_i \) in \( D \setminus (X \cup Y) \). Specifically, for each \( v \in C \), either \( v \rightarrow q_i \) or there are \( 10k \) internally disjoint 2-paths from \( v \) to \( q_i \) in \( D \setminus (X \cup Y) \). If \( v \rightarrow q_i \),  then  $p_ivq_i$ is a desired path internally disjoint from \( V(\mathcal{Q}^{**}) \cup V(\mathcal{P}) \).   Otherwise, there exist \( 10k \) internally disjoint 2-paths from \( v \) to \( q_i \) in \( D \setminus (X \cup Y) \).  Among these paths, at most \( 3k \) intersect \( V(\mathcal{Q}^{**}) \). Since otherwise, there is such a path $vwq_i$ such that $w\in Q_j$ is at least the forth vertex on $Q_j$. Since \( v \in C \subseteq B \), there exists \( u \in U \setminus \mathrm{Init}(\mathcal{Q}^{**}) \) with \( u \to v \). Then $uvw\circ Q_j[v,y_j]$ is a shorter path than $Q_j$, which is disjoint from the other paths in $\mathcal{Q}^{**}\setminus Q_j$, a contradiction.  Additionally, among the $10k$ paths, there are at most \(3k\) that intersect with \(V(\mathcal{P})\) (as these 2-paths in \( D \setminus (X \cup Y) \) are disjoint with \( X \), i.e., Init$(\mathcal{P})$). Hence, there are still at least $2k$ such 2-paths internally disjoint with \( V(\mathcal{Q}^{**}) \cup V(\mathcal{P}) \cup C \). Therefore, there are at least $2k$ internally disjoint $(p_i,q_i)$-paths of length at most 3 in $ (D \setminus  {(}V(\mathcal{Q}^{**}) \cup V(\mathcal{P})) )\cup \{p_i, q_i\}$, the proof is complete.   
\end{proof}

{Claim \ref{cla5} allows us to construct {a set $\mathcal{R}$ of} $k$ disjoint paths of lengths at most 3 linking $p_i$ to $q_i$ for $i \in [k]$ in $(D \setminus  {(}V(\mathcal{Q}^{**}) \cup V(\mathcal{P})) )\cup \{p_1, \ldots, p_k, q_{1}, \ldots, q_k\}$. Indeed, assuming we have constructed the paths $R_1,\ldots,R_i$ $(i<k)$, then we have $|\text{Int}(R_1)\cup\cdots\cup \text{Int}(R_i)|\leq 2i<2k$, there also exists a $(p_{i+1},q_{i+1})$-path $R_{i+1}$ disjoint from $V(R_1)\cup\cdots\cup V(R_i)\cup\{p_{i + 2},\ldots,p_k\cup\{q_{i+2},\ldots,q_{k}\}$. Then repeat this process until we have the required $k$ paths, which concludes the proof.}
\hfill $\square$

%\section{Remark}

\end{document}